\documentclass[a4paper,reqno,10pt]{amsart}
\usepackage{amsfonts,amssymb,amsthm,amsbsy,latexsym,amscd,amsmath,euscript,enumerate,manfnt, marvosym,verbatim,calc,mathrsfs,marvosym,tensor,textcomp ,color,xcolor,tikz,setspace,upgreek,bbm,bm,datetime}
\usepackage[linkcolor=blue,colorlinks=true]{hyperref}

\newtheorem{theorem}{Theorem}[section]
\newtheorem{corollary}[theorem]{Corollary}

\newtheorem{proposition}[theorem]{Proposition}

\theoremstyle{definition}
\newtheorem{definition}[theorem]{Definition}
\newtheorem{construction}[theorem]{Construction}
\newtheorem{remark}[theorem]{Remark}
\newtheorem*{conjecture}{Conjecture}

\newtheorem{example}[theorem]{Example}
\newtheorem*{acknowledgement}{Acknowledgement}

\newcommand{\tr}[2][]{\ensuremath{{\textnormal{tr}}_{#1}(#2)}}
\newcommand{\uk}[2][]{\ensuremath{{\textnormal{k}}_{#1}(#2)}}

\newcommand{\MIF}[2][]{\ensuremath{{\textnormal{MIF}}_{#1}(#2)}}
\newcommand{\CIF}[2][]{\ensuremath{{\textnormal{CIF}}_{#1}(#2)}}

\newcommand{\M}[2][]{\ensuremath{{\textnormal{M}}_{#1}(#2)}}

\makeatletter
\def\imod#1{\allowbreak\mkern10mu({\operator@font mod}\,\,#1)}
\def\@textbottom{\vskip\z@\@plus 18pt}
\let\@texttop\relax
\makeatother

\title[Closed Intersecting Family of $k-$sets]{\textnormal{Closed Intersecting Families of finite sets\\ and their applications}}
\author{Kaushik Majumder}
\address{\newline Theoretical Statistics and Mathematics Unit\\ \newline Indian Statistical Institute, Bangalore Centre \\ \newline $8^{th}$ Mile Mysore Road, Bangalore - $560059$, India.\newline \textnormal{\textestimated-Mail}: {\tt kaushik\MVAt isibang.ac.in}}

\date{6th December,2014.}

\begin{document}

\begin{abstract}
Paul Erd\H{o}s and L\'{a}szl\'{o} Lov\'{a}sz established that any \emph{maximal intersecting family of $k-$sets} has at most $k^{k}$ blocks. They introduced the problem of finding the maximum possible number of blocks in such a family.
They also showed that there exists a maximal intersecting family of $k-$sets with approximately $(e-1)k!$ blocks. Later P\'{e}ter Frankl, Katsuhiro Ota and Norihide Tokushige used a remarkable construction to prove the existence of a maximal intersecting family of $k-$sets with at least $(\frac{k}{2})^{k-1}$ blocks. In this article we introduce the notion of a \emph{closed intersecting family of $k-$sets} and show that such a family can always be embedded in a maximal intersecting family of $k-$sets. Using this result we present two examples which disprove two special cases of one of the conjectures of Frankl et al. This article also provides comparatively simpler construction of maximal intersecting families of $k-$sets with at least $(\frac{k}{2})^{k-1}$ blocks.
\end{abstract}

\maketitle

\section{Introduction}
 
By a family we mean a family (set) of finite sets. Such a family is called \emph{intersecting} if any two of its member have non empty intersection. A \emph{maximal intersecting family} is an intersecting family which can not be embedded properly into any larger intersecting family. In this article, our idea is to decompose a maximal intersecting family into some suitable subfamilies and study these subfamilies to gain a better understanding of such a family. Using this idea we are able to locate a similarity between the recursive Erd\H{o}s-Lov\'{a}sz construction in \cite[Construction~(c), Page~620]{MR0382050} and a non recursive Frankl-Ota-Tokushige construction in \cite[\textsection~2, Example~1 \& Example~2]{MR1383503}. We find that each maximal intersecting family has a ``core'' which generates it. We call this core a \emph{closed intersecting family}. In \cite{MR1383503}, Frankl et al. conjectured that the maximal intersecting family of $k-$sets constructed by them has the largest number of blocks, and 
it is the only such family (up to isomorphism)  with these many blocks. We use the theory developed here to prove that both these conjectures are false, at least for small $k$. (See Example~\ref{FOT(4,2)} and Example~\ref{FOT(4&5,3)} below.) Before going into the technicalities let us give some notations and definitions.

\subsection{Notations and Terminologies :} Let $\mathcal{G}$ and $\mathcal{H}$ be two non empty families of non empty sets. 
\begin{enumerate}[(a)]
\item $A\sqcup B$ denotes the union of two disjoint sets $A$ and $B$.
\item $\mathcal{G}\sqcup \mathcal{H}$ denotes the union of two disjoint families $\mathcal{G}$ and $\mathcal{H}$.
\item For any set $A$, $|A|$ will denote the cardinality of $A$.
\item Any $B\in\mathcal{G}$ is called a \emph{block} of $\mathcal{G}$.
\item A family $\mathcal{G}$ is said to be \emph{uniform} if all its blocks have the same size. If $\mathcal{G}$ is a uniform family we shall denote its common block size by $\uk{\mathcal{G}}$.
\item The \emph{point set} of the family $\mathcal{G}$ is defined as $\underset{B\in\mathcal{G}}{\cup} B$ and is denoted by $P_{\mathcal{G}}$. Any $x\in P_{\mathcal{G}}$ is called a \emph{point} of $\mathcal{G}$. 
\item Suppose $P_{\mathcal{G}}$ and $P_{\mathcal{H}}$ are disjoint. Then $\mathcal{G}\circledast\mathcal{H}$ denotes the collection of all sets of the form $A\sqcup B$, where $A\in\mathcal{G}$ and $B\in\mathcal{H}$. If $\mathcal{G}$ consists of a single $k-$set $B$, then we denote $\mathcal{G}\circledast\mathcal{H}$ by $B\circledast\mathcal{H}$. If $\mathcal{G}$ consists of a single $1-$set $\{\alpha\}$, then we denote $\mathcal{G}\circledast\mathcal{H}$ by $\alpha\circledast\mathcal{H}$.
\item A family $\mathcal{G}$ is said to be \emph{isomorphic} to the family $\mathcal{H}$ if there exists a one-to-one and onto function $\phi: P_{\mathcal{G}}\rightarrow P_{\mathcal{H}}$ such that $\phi(B)\in\mathcal{H}$ if and only if $B\in\mathcal{G}$. 
\item  A \emph{blocking set} of a family $\mathcal{G}$ is a set $C$ which intersects every block of $\mathcal{G}$. A \emph{transversal} of $\mathcal{G}$ to be a blocking set of $\mathcal{G}$ with the smallest possible size \--- in case $\mathcal{G}$ has a finite blocking set. In this case we denote the common size of its transversals by $\tr{\mathcal{G}}$. If $\mathcal{G}$ has no finite blocking set, we may put $\tr{\mathcal{G}}=\infty$. If $\tr{\mathcal{G}}<\infty$, we denote the family of transversals of $\mathcal{G}$ by $\mathcal{G}^{\top}$. Note that $\mathcal{G}^{\top}$ is a uniform family with $\uk{\mathcal{G}^{\top}}=\tr{\mathcal{G}}$.
\end{enumerate}

\begin{definition}
A family $\mathcal{F}$ is said to be a \emph{maximal intersecting family} (in short MIF) if $\tr{\mathcal{F}}<\infty$ and $\mathcal{F}=\mathcal{F}^{\top}$. We use $\MIF{k}$ as a generic name for uniform MIF's with $\uk{\mathcal{F}}=k$.
\end{definition}

Erd\H{o}s and Lov\'{a}sz established, in their landmark article \cite{MR0382050}, that any $\MIF{k}$ has at most $k^{k}$ blocks. They showed by means of an example that there exists a $\MIF{k}$ with approximately $(e-1)k!$ blocks. This example is constructed by a recursive procedure \cite[Construction~(c), Page~620]{MR0382050} starting with the unique $\MIF{1}$. Lov\'{a}sz conjectured in \cite{MR510823}, that the $\MIF{k}$ thus constructed was the extremal one. Later in \cite{MR1383503}, an extremely elegant and complicated example was given to show that there exists a $\MIF{k}$ with at least (approximately) $(\frac{k}{2})^{k-1}$ blocks (i.e. it has the largest number of blocks) and it disproves Lov\'{a}sz conjecture. In this article, we present comparatively simpler construction (see $\mathbb{G}(k,t)$ in Construction~\ref{circular_construction}) to prove that there exists a $\MIF{k}$ with at least (approximately) $(\frac{k}{2})^{k-1}$ blocks. (More precisely, we present an alternative proof of 
\cite[\textsection~2, Theorem~1]{MR1383503}, see Corollary~\ref{beating_LovaszConjecture} below). In  \cite{MR1383503}, it is conjectured that the construction of Frankl et al. yields the unique $\MIF{k}$ with the largest number of blocks. Here we show that both parts of this conjecture are false. Specifically, the uniqueness part is incorrect for $k=4$, while the optimality part is incorrect for $k=5$. 

\subsection{Organisation of this article} Section~\ref{CIF} of this article is the introductory part about closed intersecting families, some of its properties and examples. In Section~\ref{circular_constructions} we study  constructions over the cycle graph. In Section~\ref{beating} it is shown that Example~\ref{FOT(4,2)} and Example~\ref{FOT(4&5,3)} are counter examples to \cite[\textsection~3, Conjecture~4]{MR1383503} in special cases. In this final section we close this article by stating a conjecture.

\section{Closed Intersecting Family of finite sets}\label{CIF}

In this section, we present results of an ab initio study on closed intersecting families. 

\begin{definition}
Let $\mathcal{F}$ be a uniform family with $\uk{\mathcal{F}}=k$ and $\tr{\mathcal{F}}=t$. $\mathcal{F}$ is said to be a \emph{closed intersecting family} (in short CIF) if $\tr{\mathcal{F}}\leq\uk{\mathcal{F}}-1$ and $\mathcal{F}=(\mathcal{F}\sqcup\mathcal{F}^{\top})^{\top}$. We use $\CIF{k,t}$ as a generic name for CIF's $\mathcal{F}$ with $\uk{\mathcal{F}}=k$ and $\tr{\mathcal{F}}=t$. Note that any closed intersecting family is necessarily an intersecting family.
\end{definition}

We have the following characterisation.  

\begin{proposition}\label{equivalentclosure}
Let $\mathcal{F}$ be an intersecting family of $k-$sets with $\tr{\mathcal{F}}\leq k-1$. Then the following statements are equivalent:
\begin{enumerate}[\normalfont(a)]
\item\label{descriptional} Any $k-$set which is a blocking set of $\mathcal{F}\sqcup\mathcal{F}^{\top}$ is a block of $\mathcal{F}$.
\item\label{applicational} Any $k-$set which is a blocking set of $\mathcal{F}$ but itself is not a block of $\mathcal{F}$, then it is not a blocking set of $\mathcal{F}^{\top}$.
\item\label{representational} $\mathcal{F}=(\mathcal{F}\sqcup\mathcal{F}^{\top})^{\top}$.
\end{enumerate}
\end{proposition}
\begin{proof}
Firstly we prove \eqref{descriptional} $\Leftrightarrow$ \eqref{applicational} and then we prove \eqref{representational} $\Leftrightarrow$ \eqref{descriptional}.

Let $C$ be a $k-$set which is a blocking set of $\mathcal{F}$ and $C\notin\mathcal{F}$. Suppose $C$ is a blocking set of $\mathcal{F}^{\top}$, then by \eqref{descriptional} $C\in\mathcal{F}$, a contradiction. Hence $C$ is not a blocking set of $\mathcal{F}^{\top}$. Conversely, let $C$ be a $k-$set which is a blocking set of $\mathcal{F}\sqcup\mathcal{F}^{\top}$. Suppose $C\notin\mathcal{F}$, then by \eqref{applicational} $C$ is a blocking set of $\mathcal{F}^{\top}$, a contradiction to the assumption, so our supposition $C\notin\mathcal{F}$ was wrong. Hence $C\in\mathcal{F}$.

From \eqref{representational} it implies that $\tr{\mathcal{F}\sqcup\mathcal{F}^{\top}}=k$. Let $C$ be a blocking $k-$set of $\mathcal{F}\sqcup\mathcal{F}^{\top}$, then $C$ is transversal of the family $\mathcal{F}\sqcup\mathcal{F}^{\top}$. Hence $C\in\mathcal{F}$. Conversely, let $C$ be a transversal of $\mathcal{F}\sqcup\mathcal{F}^{\top}$. Suppose $|C|\leq k-1$. Consider a set $X$,  of size $k-|C|$, disjoint from $P_{\mathcal{F}}$. Then $X\sqcup C$ is a blocking $k-$set of $\mathcal{F}\sqcup\mathcal{F}^{\top}$ and it is not a block of $\mathcal{F}$, a contradiction to \eqref{descriptional}. So $|C|=k$ and hence by \eqref{descriptional} $C\in\mathcal{F}$, which proves \eqref{representational}.
\end{proof}

Henceforth, by \emph{closure property} we refer any one of \eqref{descriptional}, \eqref{applicational} and \eqref{representational} in our study.

\begin{example}
Let $k,t$ be positive integers with $t\leq k-1$. All $k-$subsets of a $(k+t-1)-$set form a $\CIF{k,t}$ and its transversals are all $t-$subsets of the point set. It has $k+t-1$ points, $\binom{k+t-1}{k}$ blocks and $\binom{k+t-1}{t}$ transversals.
\end{example}

\begin{example}
Let $k,t$ be positive integers with $2\leq t\leq k-1$. Let $P$ be a $(k+t-2)-$set. For each bi partition 
$(C, P\smallsetminus C)$ of $P$ with $|C|=t-1$, we introduce a new symbol $x_{c}$. We consider the family of all $k-$subsets of $P$ plus all $k-$sets of the form $\{x_{c}\}\sqcup(P\smallsetminus C)$. It is a $\CIF{k,t}$. Its transversals are all $t-$subsets of $P$ plus all $t-$sets of the form $\{x_{c}\}\sqcup C$. It has $k+t-2+\binom{k+t-2}{k-1}$ points, $\binom{k+t-2}{k}+\binom{k+t-2}{k-1}$ blocks and $\binom{k+t-2}{t}+\binom{k+t-2}{t-1}$ transversals.
\end{example}

\begin{theorem}\label{inverse}
Let $\mathcal{F}$ be a subfamily of a $\MIF{k}$ $\mathcal{X}$ such that $t:=\tr{\mathcal{F}}\leq k-1$ and $\mathcal{X}\smallsetminus\mathcal{F}=\mathcal{A}\circledast\mathcal{F}^{\top}$ for some family $\mathcal{A}$. Then $\mathcal{F}$ is a $\CIF{k,t}$ if and only if $\mathcal{A}$ is a $\MIF{k-t}$. 
\end{theorem}
\begin{proof}
Let $\mathcal{F}$ be a $\CIF{k,t}$ and let $T\in\mathcal{F}^{\top}$. Then, by assumption, $|T|=t\leq k-1$. Hence there exists at least one $T^{'}\in\mathcal{F}\sqcup\mathcal{F}^{\top}$ disjoint from $T$. Since $T^{'}\cap T=\emptyset$ and $T\in\mathcal{F}^{\top}$, it follows that $T^{'}\notin\mathcal{F}$. So $T^{'}\in\mathcal{F}^{\top}$. Thus for each $T\in\mathcal{F}^{\top}$ there exists at least one $T^{'}\in\mathcal{F}^{\top}$ with $T^{'}\cap T=\emptyset$. Since $\mathcal{A}\circledast\mathcal{F}^{\top}$ is an intersecting family of $k-$sets, it follows that $\mathcal{A}$ is an intersecting family of $(k-t)-$sets.

Let $C\in\mathcal{A}^{\top}$. Then, for each $T\in\mathcal{F}^{\top}$, $C\sqcup T$ is a blocking set of $\mathcal{X}$ and hence $|C\sqcup T|\geq k$. Thus $|C|\geq k-t$, with equality if $C\sqcup T\in\mathcal{X}$ for all $T\in\mathcal{F}^{\top}$. Since each block of $\mathcal{A}$ is a blocking set of size $k-t$ for $\mathcal{A}$, it follows that $\tr{\mathcal{A}}=k-t$ and $\mathcal{A}\subseteqq\mathcal{A}^{\top}$. Also if $C\in\mathcal{A}$ and $T\in\mathcal{F}^{\top}$, then $C\sqcup T\in\mathcal{X}$. The argument in the previous paragraph shows that $C\sqcup T$ is not a blocking set of $\mathcal{F}^{\top}$ and hence 
$C\sqcup T\notin(\mathcal{F}\sqcup\mathcal{F}^{\top})^{\top}=\mathcal{F}$. Thus 
$C\sqcup T\in\mathcal{X}\smallsetminus\mathcal{F}=\mathcal{A}\circledast\mathcal{F}^{\top}$. Hence $C\in\mathcal{A}$. Thus $\mathcal{A}^{\top}\subseteqq\mathcal{A}$ and hence $\mathcal{A}=\mathcal{A}^{\top}$. Thus $\mathcal{A}$ is a $\MIF{k-t}$.

Conversely, suppose $\mathcal{A}$ is a $\MIF{k-t}$. Since $\mathcal{F}$ is an intersecting family of $k-$sets, every block of $\mathcal{F}$ is a blocking $k-$set of $\mathcal{F}\sqcup\mathcal{F}^{\top}$ and hence $\tr{\mathcal{F}\sqcup\mathcal{F}^{\top}}\leq k$. To show that $\mathcal{F}$ is a $\CIF{k,t}$, it suffices to show that if $C$ is a blocking set of size $k$ for $\mathcal{F}$ which is not a block of $\mathcal{F}$, then $C$ is not a blocking set of $\mathcal{F}^{\top}$. If $C\notin\mathcal{A}\circledast\mathcal{F}^{\top}$, then $C$ is not a block of $\mathcal{X}$ and hence there is a block $B\in\mathcal{X}$ disjoint from $C$. But $C$ is a blocking set of $\mathcal{F}$. So $B\in\mathcal{A}\circledast\mathcal{F}^{\top}$. Then $B\cap P_{\mathcal{F}}$ is a block of $\mathcal{F}^{\top}$ disjoint from $C$, so that $C$ is not a blocking set of $\mathcal{F}^{\top}$ in this case. On the other hand, if $C\in\mathcal{A}\circledast\mathcal{F}^{\top}$, then choose a point $\alpha\in C\cap P_{\mathcal{A}}$. (It 
exists since $k=|C|>t=\uk{\mathcal{F}^{\top}}$.) Since $C$ is a block of a $\MIF{k}$ $\mathcal{X}$, there exists at least one $B\in\mathcal{X}$ such that $B\cap C=\{\alpha\}$. Since $\alpha\notin P_{\mathcal{F}}$, it follows that $B\in\mathcal{A}\circledast\mathcal{F}^{\top}$ and hence $B\cap P_{\mathcal{F}}$ is a block of $\mathcal{F}^{\top}$ disjoint from $C$. So $C$ is not a blocking set of $\mathcal{F}^{\top}$ in this case also.  
\end{proof}

\begin{corollary}
Let $\mathcal{X}$ be a $\MIF{k}$, then there exists at least one $\CIF{k,k-1}$ $\mathcal{F}$ and $\alpha\in P_{\mathcal{X}}\smallsetminus P_{\mathcal{F}}$ such that $\mathcal{X}=\mathcal{F}\sqcup\alpha\circledast\mathcal{F}^{\top}$.
\end{corollary}
\begin{proof}
Let $\alpha \in P_{\mathcal{X}}$. Define $\mathcal{F}=\{B \in \mathcal{X}: \alpha \notin B\}$. Then $\mathcal{F}^{\top}=\{B\smallsetminus\{\alpha\}: \alpha\in B\in \mathcal{X}\}$. The conclusion follows as an application of Theorem~\ref{inverse}. 
\end{proof}

The following theorem is a sort of converse to Theorem~\ref{inverse}. Together, Theorem~\ref{inverse} and Theorem~\ref{embedding_theorem} show that closed intersecting families are the cores which may be used to obtain recursive construction of maximal intersecting families. 

\begin{theorem}\label{embedding_theorem}
Let $\mathcal{A}$ and $\mathcal{F}$ be a $\MIF{k-t}$ and a $\CIF{k,t}$ respectively where $\mathcal{A}$ and $\mathcal{F}$ have disjoint point sets. Then $\mathcal{F}\sqcup\mathcal{A}\circledast\mathcal{F}^{\top}$ is a $\MIF{k}$.    
\end{theorem}
\begin{proof}
Let $C$ be a blocking $k-$set of $\mathcal{F}\sqcup\mathcal{A}\circledast\mathcal{F}^{\top}$. To prove $C\in\mathcal{F}\sqcup\mathcal{A}\circledast\mathcal{F}^{\top}$. If $C\in\mathcal{F}$ we are done. So assume $C\notin\mathcal{F}$. So by closure property of $\mathcal{F}$, $C$ is not a blocking set of $\mathcal{F}\sqcup\mathcal{F}^{\top}$. This implies $C$ is not a blocking set of $\mathcal{F}^{\top}$. Hence there exists at least one $T\in\mathcal{F}^{\top}$ which is disjoint from $C$. Since $C$ is a blocking set of $T\circledast\mathcal{A}$ it follows that $C\cap P_{\mathcal{A}}$ is a blocking set of $\mathcal{A}$. So $|C\cap P_{\mathcal{A}}|\geq k-t$. Also $C\cap P_{\mathcal{F}}$ is a blocking set of $\mathcal{F}$, hence $|C\cap P_{\mathcal{F}}|\geq t$. But $|C|=k$, so $|C\cap P_{\mathcal{A}}|=k-t$ and $|C\cap P_{\mathcal{F}}|=t$. Hence $C\cap P_{\mathcal{A}}\in\mathcal{A}$ and $C\cap P_{\mathcal{F}}\in\mathcal{F}^{\top}$. So $C\in\mathcal{A}\circledast\mathcal{F}^{\top}$. This shows that every blocking 
$k-$set of the family 
$\mathcal{F}\sqcup\mathcal{A}\circledast\mathcal{F}^{\top}$ is a block of that family. The converse is obvious.
\end{proof}

\begin{corollary}\label{c_embedding_theorem}
Let $\M{k}$ denote maximum number of blocks in a $\MIF{k}$. For each integer $k\geq t+1$, if $\mathcal{F}$ is a $\CIF{k,t}$ with $b$ blocks and $b^{\top}$ transversals, then  $\M{k}\geq b+b^{\top}\M{k-t}$.
\end{corollary}
\begin{proof}
We choose a $\MIF{k-t}$ with $\M{k-t}$ blocks so that its point set is disjoint from $P_{\mathcal{F}}$. Call it $\mathcal{A}$. The result follows since by Theorem~\ref{embedding_theorem}, $\mathcal{F}\sqcup\mathcal{A}\circledast\mathcal{F}^{\top}$ is a $\MIF{k}$ with $b+b^{\top}\M{k-t}$ blocks.
\end{proof}

\subsection{Construction of maximal intersecting families using closed intersecting families.}

\begin{proposition}\label{pointset}
Let $\mathcal{F}$ be a $\CIF{k,t}$. Suppose for each $i$, with $1\leq i\leq n$, $\mathcal{A}_{i}$ be a $\MIF{k-t}$ and $\mathcal{C}_{i}$ is a subfamily of $\mathcal{F}^{\top}$  with the following properties:
\begin{enumerate}[\normalfont(a)]
\item each $\mathcal{A}_{i}$ and $\mathcal{F}$ have disjoint point sets;
\item $\mathcal{F}^{\top}=\overset{n}{\underset{i=1}\sqcup}\mathcal{C}_{i}$;
\item each $t-$set of $\mathcal{C}_{i}$ is a blocking set of $\mathcal{F}^{\top}\smallsetminus\mathcal{C}_{i}$.
\end{enumerate}
Then $\mathcal{F}\sqcup(\overset{n}{\underset{i=1}\sqcup}\mathcal{A}_{i}\circledast\mathcal{C}_{i})$
is a $\MIF{k}$. Moreover, $n\leq\frac{1}{2}\binom{2t}{t}$. 
\end{proposition}
\begin{proof}
Let $\mathcal{G}:=\mathcal{F}\sqcup(\overset{n}{\underset{i=1}\sqcup}\mathcal{A}_{i}\circledast\mathcal{C}_{i})$. Clearly it is an intersecting family of $k-$sets. Let $C$ be a blocking set of $\mathcal{G}$ with size at most $k$. To prove $C$ is a block of it. If $C\in\mathcal{F}$ we are done. So assume $C\notin\mathcal{F}$. By closure property of $\mathcal{F}$ there exists at least one $T\in\mathcal{F}^{\top}$ such that $C\cap P_{\mathcal{F}}$ is disjoint from $T$ and $T\in\mathcal{C}_{i}$ for a unique $i$. Since $C$ is a blocking set of $T\circledast\mathcal{A}_{i}$ hence $|C\cap P_{\mathcal{A}_{i}}|\geq k-t$. Also $C\cap P_{\mathcal{F}}$ is a blocking set of $\mathcal{F}$ hence 
$|C\cap P_{\mathcal{F}}|\geq t$. This implies $|C\cap P_{\mathcal{A}_{i}}|=k-t$ and $|C\cap P_{\mathcal{F}}|=t$ hence $C\in\mathcal{A}_{i}\circledast\mathcal{C}_{i}$.

For the next part, by assumption (c) we  observe that, for each $i$ with $1\leq i\leq n$ there exists at least one pair $(T_{i},T^{'}_{i})$, where $T_{i}$, $T^{'}_{i}\in\mathcal{C}_{i}$ with $T_{i}\cap T^{'}_{i}=\emptyset$. Also for each $i$, $j$ and chosen such pairs $(T_{i}, T^{'}_{i})$ and $(T_{j}, T^{'}_{j})$, with $1\leq i<j\leq n$, we have $T_{i}\cap T^{'}_{j}\neq\emptyset$ and $T^{'}_{i}\cap T_{j}\neq\emptyset$. Hence by using \cite[\textsection~13, Problem~32]{MR537284}, we have $n\leq\frac{1}{2}\binom{2t}{t}$.
\end{proof}

The proof of Theorem~\ref{embedding_theorem} is an application of Proposition~\ref{pointset}. We observe that it is the case $n=1$ in Proposition~\ref{pointset}.

\begin{proposition}\label{affine}
Let $\mathcal{F}$ be a $\CIF{k,k-n}$. Suppose $\mathcal{F}^{\top}=\overset{n+1}{\underset{i=1}\sqcup}\mathcal{C}_{i}$, where for each $i$, with $1\leq i\leq n+1$, the subfamily $\mathcal{C}_{i}$ satisfies following properties:
\begin{enumerate}[\normalfont(a)]
\item each $\mathcal{C}_{i}$ is an intersecting family of $(k-n)-$sets;
\item whenever $i\neq j$, then for each $T\in \mathcal{C}_{i}$ there exists at least one $T^{'}\in\mathcal{C}_{j}$ such that 
$T\cap T^{'}=\emptyset$.
\end{enumerate}
Consider $\mathcal{A}_{1}$, $\mathcal{A}_{2}$, $\ldots$, $\mathcal{A}_{n+1}$ are $(n+1)-$parallel classes of an affine plane of order $n$ (provided it exists). This affine plane of order $n$ and $\mathcal{F}$ have disjoint point sets. Then $\mathcal{F}\sqcup(\overset{n+1}{\underset{i=1}\sqcup}\mathcal{A}_{i}\circledast\mathcal{C}_{i})$ is a $\MIF{k}$.
\end{proposition}
\begin{proof}
Let $\mathcal{G}:=\mathcal{F}\sqcup(\overset{n+1}{\underset{i=1}\sqcup}\mathcal{A}_{i}\circledast\mathcal{C}_{i})$. Clearly it is an intersecting family of $k-$sets. Let $P$ be the point set of this affine plane. Let $C$ be a blocking set of $\mathcal{G}$ with size at most $k$. To prove $C$ is a block of it. If $C\in\mathcal{F}$ we are done. So assume $C\notin\mathcal{F}$. By closure property of $\mathcal{F}$ there exists at least one $T\in\mathcal{F}^{\top}$ such that $C\cap P_{\mathcal{F}}$ is disjoint from $T$. Then there exists at least one $i$, with $1\leq i\leq n+1$, such that $T\in\mathcal{C}_{i}$. But $C$ is a blocking set of $T\circledast\mathcal{A}_{i}$ hence $|C\cap P|\geq n$. Also $C\cap P_{\mathcal{F}}$ is a blocking set of $\mathcal{F}$ hence 
$|C\cap P_{\mathcal{F}}|\geq k-n$. This implies $|C\cap P|=n$ and $|C\cap P_{\mathcal{F}}|=k-n$ and hence 
$C\cap P_{\mathcal{F}}\in\mathcal{F}^{\top}$. So by assumptions (a) and (b), $C\cap P_{\mathcal{F}}\in\mathcal{C}_{j}$ for some $j\neq i$. Then again by assumption (b) there exists at least one $T_{l}\in\mathcal{C}_{l}$ such that 
$C\cap P_{\mathcal{F}}$ is disjoint from $T_{l}$, for each $l$ with $l\neq j$. So $C\cap P$ is a blocking set of each such $T_{l}\circledast\mathcal{A}_{l}$. Hence $C\cap P$ is a line of $\mathcal{A}_{j}$. Therefore $C\in\mathcal{A}_{j}\circledast\mathcal{C}_{j}$.
\end{proof}

\subsection{Recursive Construction of closed intersecting families.}

\begin{theorem}\label{proiliferation}
Let $\mathcal{A}$ be a $\MIF{l}$ and let $\mathcal{F}_{x}$, $x\in P_{\mathcal{A}}$, be uniform families with pairwise disjoint point sets. Suppose $\uk{\mathcal{F}_{x}}=k$ and $\tr{\mathcal{F}_{x}}=t$ for all $x$. Put 
\begin{equation*}
\mathcal{G}=\left\{\underset{x\in A}{\sqcup}F_{x}: A\in\mathcal{A}, F_{x}\in\mathcal{F}_{x}\textup{ for all } x\in A\right\}.
\end{equation*}
Then we have the following.
\begin{enumerate}[\normalfont(a)]
\item $\mathcal{G}^{\top}=\left\{\underset{x\in A}{\sqcup}T_{x}: A\in\mathcal{A}, T_{x}\in\mathcal{F}^{\top}_{x}\textup{ for all } x\in A\right\}$. In particular, $\uk{\mathcal{G}}=kl$ and  $\tr{\mathcal{G}}=tl$.
\item If, further, each $\mathcal{F}_{x}$ is a $\CIF{k,t}$ with $\tr{\mathcal{F}^{\top}_{x}}=k$, then $\mathcal{G}$ is a $\CIF{kl,tl}$. 
\end{enumerate}
\end{theorem}
\begin{proof}
If $A\in\mathcal{A}$ and $T_{x}\in\mathcal{F}^{\top}_{x}$ for all $x\in A$, then clearly $\underset{x\in A}{\sqcup}T_{x}$ is a blocking set of $\mathcal{G}$ of size $tl$. Thus, $\tr{\mathcal{G}}\leq tl$. Let $B$ be a blocking set of $\mathcal{G}$ of size at most $tl$. For $x\in P_{\mathcal{A}}$, put $T_{x}= B\cap P_{\mathcal{F}_{x}}$. Let $A=\{x\in P_{\mathcal{A}}: |T_{x}|\geq t\}$. We have
\begin{equation}\label{PHP_equation}
\underset{x\in P_{\mathcal{A}}}{\sum}|T_{x}|=|B|\leq tl
\end{equation}
and hence $|A|\leq l$. If $A$ is not a block of the $\MIF{l}$ $\mathcal{A}$, then there is a block $A^{'}$ of $\mathcal{A}$ disjoint from $A$. Hence $|T_{x}|\leq t-1$ for all $x\in A^{'}$. So, for each $x\in A^{'}$ there is a block $F_{x}$ of $\mathcal{F}_{x}$ disjoint from $T_{x}$. Hence $\underset{x\in A^{'}}{\sqcup}F_{x}$ is a block of $\mathcal{G}$ disjoint from $B$, a contradiction. So $A\in\mathcal{A}$ and $|A|=l$. Then \eqref{PHP_equation} implies that $|T_{x}|=\emptyset$ for $x\notin A$ and $|T_{x}|=t$ for $x\in A$. Thus, $|B|=tl$, so that $\tr{\mathcal{G}}=tl$ and $B\in\mathcal{G}^{\top}$. Since $B=\underset{x\in A}{\sqcup}T_{x}\in\mathcal{G}^{\top}$ and $|T_{x}|=t$, it follows that $T_{x}\in\mathcal{F}^{\top}_{x}$ for all $x\in A$. This proves part (a).

Now we assume each $\mathcal{F}_{x}$ is a $\CIF{k,t}$. Since $\mathcal{A}$, as well as each $\mathcal{F}_{x}$, is an intersecting family it follows that $\mathcal{G}$ is an intersecting family. Using the description of $\mathcal{G}^{\top}$ from part (a) and applying part (a) to the families $\mathcal{F}^{\top}_{x}$, $x\in P_{\mathcal{A}}$, we see that 
\begin{equation*}
\mathcal{G}^{\top\top}=\left\{\underset{x\in A}{\sqcup}S_{x}: A\in\mathcal{A}, S_{x}\in\mathcal{F}^{\top\top}_{x}\textup{ for all } x\in A\right\}.
\end{equation*}
Thus $\tr{\mathcal{G}\sqcup\mathcal{G}^{\top}}\geq\tr{\mathcal{G}^{\top}}=kl$. Since all the blocks of $\mathcal{G}$ are blocking sets of $\mathcal{G}\sqcup\mathcal{G}^{\top}$ of size $kl$, it follows that $\tr{\mathcal{G}\sqcup\mathcal{G}^{\top}}=kl$ and $\mathcal{G}\subseteq(\mathcal{G}\sqcup\mathcal{G}^{\top})^{\top}$. Let $C$ be a transversal of $\mathcal{G}\sqcup\mathcal{G}^{\top}$. Then $C\in\mathcal{G}^{\top\top}$ and hence $C=\underset{x\in A}{\sqcup}S_{x}$, for some $A\in\mathcal{A}$ and $S_{x}\in\mathcal{F}^{\top\top}_{x}$ for all $x\in A$. If we can show that $C\in\mathcal{G}$, then we are done. Otherwise, there exists at least one $y\in A$ such that $S_{y}\notin\mathcal{F}_{y}=(\mathcal{F}_{y}\sqcup\mathcal{F}^{\top}_{y})^{\top}$. Since $S_{y}\in\mathcal{F}^{\top\top}_{y}$ it follows that $S_{y}$ is not a blocking set of $\mathcal{F}_{y}$. So there exists at least one $U_{y}\in\mathcal{F}_{y}$ disjoint from $S_{y}$. Since $y\in A\in\mathcal{A}$ and $\mathcal{A}$ is a $\MIF{l}$, there is a  
$B\in\mathcal{A}$ such that $A\cap B=\{y\}$. For each $x\in B\smallsetminus\{y\}$, choose arbitrary $U_{x}\in\mathcal{F}_{x}$. Then $\underset{x\in B}{\sqcup}U_{x}$ is a block of $\mathcal{G}$ disjoint from the blocking set $C$, a contradiction. Thus $C\in\mathcal{G}$. Hence $(\mathcal{G}\sqcup\mathcal{G}^{\top})^{\top}\subseteq\mathcal{G}$. Therefore $\mathcal{G}=(\mathcal{G}\sqcup\mathcal{G}^{\top})^{\top}$ and this proves part (b).
\end{proof}

\begin{theorem}\label{extension&embeding}
Let $\mathcal{F}$ and $\mathcal{G}$ be two uniform families with disjoint point sets. Let $\uk{\mathcal{F}}=k$, $\uk{\mathcal{G}}=k+t$, $\tr{\mathcal{F}}=t^{'}$ and $\tr{\mathcal{G}}=t$. Suppose $\tr{\mathcal{G}^{\top}}>t+t^{'}$. Let $\mathcal{H}=\mathcal{G}\sqcup(\mathcal{F}\circledast\mathcal{G}^{\top})$. Then,
\begin{enumerate}[\normalfont(a)]
\item $\mathcal{H}^{\top}=\mathcal{F}^{\top}\circledast\mathcal{G}^{\top}$. In particular $\uk{\mathcal{H}}=k+t$, $\tr{\mathcal{H}}=t+t^{'}$.
\item If, further, both $\mathcal{F}$ and $\mathcal{G}$ are closed intersecting families, then $\mathcal{H}$ is a closed intersecting family.
\end{enumerate}
\end{theorem}
\begin{proof}
Since every member of $\mathcal{F}^{\top}\circledast\mathcal{G}^{\top}$ is a blocking set of $\mathcal{H}$ of size $t+t^{'}$, $\tr{\mathcal{H}}\leq t+t^{'}$. Let $C$ be a transversal of $\mathcal{H}$. Then $|C|\leq t+t^{'}$. If $C\cap P_{\mathcal{F}}$ is not a blocking set of $\mathcal{F}$, then there is a block $A\in\mathcal{F}$ disjoint from $C\cap P_{\mathcal{F}}$. Since $|C|\leq t+t^{'}\leq\tr{\mathcal{G}^{\top}}-1$, there is a $B\in\mathcal{G}^{\top}$ disjoint from $C\cap P_{\mathcal{G}}$. Then $A\sqcup B\in\mathcal{H}$ is disjoint from the blocking set $C$, a contradiction. Thus, $C\cap P_{\mathcal{F}}$ is a blocking set of $\mathcal{F}$. Clearly $C\cap P_{\mathcal{G}}$ is a blocking set of $\mathcal{G}$. Therefore $|C\cap P_{\mathcal{F}}|\geq t^{'}$ and $|C\cap P_{\mathcal{G}}|\geq t$. Since $|C|\leq t+t^{'}$, it follows that $|C\cap P_{\mathcal{F}}|=t^{'}$ and $|C\cap P_{\mathcal{G}}|=t$. Therefore
$C\cap P_{\mathcal{F}}\in\mathcal{F}^{\top}$ and $C\cap P_{\mathcal{G}}\in\mathcal{G}^{\top}$. Thus $C\in\mathcal{F}^{\top}\circledast\mathcal{G}^{\top}$. This proves part (a).

Now suppose $\mathcal{F}$ and $\mathcal{G}$ are closed intersecting families. In particular they are intersecting families. Hence $\mathcal{H}$ is an intersecting family. Thus the blocks of $\mathcal{H}$ are blocking sets of $\mathcal{H}\sqcup\mathcal{H}^{\top}$ of size $k+t$. So $\tr{\mathcal{H}\sqcup\mathcal{H}^{\top}}\leq k+t$. Let $C$ be a transversal of $\mathcal{H}\sqcup\mathcal{H}^{\top}$. Thus $|C|\leq k+t$. If we can show that $C\in\mathcal{H}$, then $\mathcal{H}$ is a closed intersecting family and we are done. If $C\in\mathcal{G}$ we are done. So suppose $C\notin\mathcal{G}$. But $C$ is a blocking set of $\mathcal{G}$. Since $\mathcal{G}$ is a closed intersecting family, it follows that there is a $T\in\mathcal{G}^{\top}$ disjoint from $C$. Since $C$ is a blocking set of $\mathcal{F}\circledast\mathcal{G}^{\top}\subseteq\mathcal{H}$ and also of $\mathcal{F}^{\top}\circledast\mathcal{G}^{\top}=\mathcal{H}^{\top}$, it follows that $C\cap P_{\mathcal{F}}$ is a blocking set of 
$\mathcal{F}\sqcup\mathcal{F}^{\top}$. Since $\mathcal{F}$ is a closed intersecting family with $\uk{\mathcal{F}}=k$, we get $|C\cap P_{\mathcal{F}}|\geq k$. Also, as $C\cap P_{\mathcal{G}}$ is a blocking set of $\mathcal{G}$ and $\tr{\mathcal{G}}=t$, $|C\cap P_{\mathcal{G}}|\geq t$. But $|C|\leq k+t$. Therefore, $|C|=k+t$, 
$C\cap P_{\mathcal{F}}\in\mathcal{F}$, $C\cap P_{\mathcal{G}}\in\mathcal{G}^{\top}$. Consequently,  $C\in\mathcal{F}\circledast\mathcal{G}^{\top}\subseteq\mathcal{H}$. Hence $C\in\mathcal{H}$.  This proves part (b).   
\end{proof}

\section{Constructions over the Cycle Graph}\label{circular_constructions}

\begin{construction}\label{circular_construction}
Let $k,t$ be positive integers with $t\leq k$. Let $X_{n}$, $0\leq n\leq t-1$, be $t$ pairwise disjoint sets with  
\begin{align*}
|X_{n}|=\left\{\begin{array}{lcr}
                k-\lfloor\frac{t}{2}\rfloor & \textnormal{if} & 0\leq n\leq\lfloor\frac{t-1}{2}\rfloor\\
                k-\lfloor\frac{t-1}{2}\rfloor & \textnormal{if} & \lfloor\frac{t-1}{2}\rfloor+1\leq n\leq t-1
            \end{array}\right.
\end{align*}
say $X_{n}=\{x^{n}_{p}:0\leq p\leq |X_{n}|-1\}$. Let $\mathbb{F}(k,t)$ be the family of all the $k-$sets of the form 
\begin{equation*}
X_{n}\sqcup\left\{x^{n+i}_{p_{i}}:1\leq i\leq k-|X_{n}|\right\},
\end{equation*}
where addition in the superscript is modulo $t$ and the sequence $\{p_{n}:n\geq 0\}$ defined as $p_{0}=0$ and for $n\geq1$, $p_{n-1}\leq p_{n}\leq 1+p_{n-1}$ (i.e. $p_{n}$ assigns only $p_{n-1}$ and $1+p_{n-1}$).  
Let $\mathbb{G}(k,t)$ be the family of all the $k-$sets of the form 
\begin{equation*}
X_{n}\sqcup\left\{x^{n+i}_{p}\in X_{n+i}:1\leq i\leq k-|X_{n}|\right\},
\end{equation*}
(i.e. $X_{n}$ is unionised with one element from each $X_{n+i}$) where addition in the superscript and subscript is modulo $t$.
\end{construction}

Clearly, both the families $\mathbb{F}(k,t)$ and $\mathbb{G}(k,t)$ are examples of intersecting family of $k-$sets (since the $t-$cycle is a graph with diameter $\lfloor\frac{t}{2}\rfloor$). It is not hard to see that the family $\mathbb{F}(k,t)$ defined above is a subfamily of $\mathbb{G}(k,t)$. In this context, we emphasise that the   
family $\mathscr{G}$ of \cite[\textsection~2, Example~1 \& Example~2]{MR1383503} is not isomorphic to $\mathbb{F}(k,t)$ .

\begin{theorem}
$\tr{\mathbb{G}(k,t)}=t$.
\end{theorem}
\begin{proof}
We prepare a $t-$set $B$ by choosing one element from each $X_{n}$, with $0\leq n\leq t-1$, then $B$ is a blocking set of $\mathbb{G}(k,t)$. Therefore $\tr{\mathbb{G}(k,t)}\leq t$. Let $C$ be an arbitrary but fixed set of size $t-1$, to show $\tr{\mathbb{G}(k,t)}\geq t$, it is enough to show there exists a block of $\mathbb{G}(k,t)$, which is disjoint from it. We divide our arguments in the following two exhaustive cases.

\noindent{\tt Case~A :} For each $n$, with $0\leq n\leq t-1$, $|C\cap X_{n}|\leq|X_{n}|-1$.

Since $|C|=t-1$ there exists $X_{n}$, with $0\leq n\leq t-1$, which is disjoint from $C$, call such an $n=n_{0}$. For this case we have, for each $m$, with $1\leq m\leq  k-|X_{n_{0}}|$, $X_{m}\smallsetminus C$ is non empty and choose one element namely, $x^{m}_{p_{m}}\in X_{m}\smallsetminus C$. Therefore, $X_{n_{0}}\sqcup\{x^{n_{0}+i}_{p_{n_{0}+i}}\in X_{n_{0}+i}\smallsetminus C:1\leq i\leq  k-|X_{n_{0}}|\}$ is the required block of $\mathbb{G}(k,t)$, which is disjoint from $C$.

\noindent{\tt Case~B :} For some $n$, with $0\leq n\leq t-1$, $C\cap X_{n}=X_{n}$. (This case will arise for $k$, with $t\leq k\leq t-1+\lfloor\frac{t-1}{2}\rfloor$.)

Since $|C|=t-1$ so there exists at most one $n$, with $0\leq n\leq t-1$, such that $C\cap X_{n}=X_{n}$. So
\begin{equation*}
|C\smallsetminus X_{n}|=t-1-|X_{n}|\leq\left\lfloor\frac{t}{2}\right\rfloor-1-k.
\end{equation*}
Since $k\geq t$ we have, 
\begin{equation*}
|C\smallsetminus X_{n}|=t-1-|X_{n}|\leq\left\lfloor\frac{t}{2}\right\rfloor-1.
\end{equation*}
So there exists at least one $m$, with $n+1\leq m\leq n+\lfloor\frac{t}{2}\rfloor$, so that $(C\smallsetminus X_{n})\cap X_{m}$ is empty, call such an $m=m_{0}$. Therefore, for $1\leq i\leq k-|X_{m_{0}}|$, we have $X_{m_{0}+i}\smallsetminus C$ is non empty and choose one element namely, $x^{m_{0}+i}_{p_{m_{0}+i}}\in X_{m_{0}+i}\smallsetminus C$. Therefore, $X_{m_{0}}\sqcup\{x^{m_{0}+i}_{p_{m_{0}+i}}\in X_{m_{0}+i}\smallsetminus C:1\leq i\leq k-|X_{m_{0}}|\}$ is the required block of $\mathbb{G}(k,t)$, which is disjoint from $C$. 
\end{proof}

\begin{theorem}\label{G_closure}
For $k\geq t+1$, $\mathbb{G}(k,t)$ is a $\CIF{k,t}$. Moreover for each $n$, with $0\leq n\leq t-1$,  
$|X_{n}\cap T|=1$, where $T$ is a transversal of $\mathbb{G}(k,t)$.
\end{theorem}
\begin{proof}
Let $C$ be a $k-$set. If for each $n$, with $0\leq n\leq t-1$, $C\cap X_{n}\subsetneqq X_{n}$ then $X_{n}\smallsetminus C$ is non empty and $T(C):=\{x_{n}\in X_{n}\smallsetminus C:0\leq n\leq t-1\}$ is a transversal of $\mathbb{G}(k,t)$, which is disjoint from $C$. So suppose for some $n$, with $0\leq n\leq t-1$, $C\cap X_{n}=X_{n}$; since $|C|=k$ and $k\geq t+1$ therefore there exists at most one such $n$, call it $n_{0}$. Therefore $|C\smallsetminus X_{n_{0}}|=k-|X_{n_{0}}|$. We observe that for each $m\neq n_{0}$, with $0\leq m\leq t-1$,  $C\cap X_{m}\subsetneqq X_{m}$, hence $X_{m}\smallsetminus C$ is non empty and choose $x^{m}_{q_{m}}\in X_{m}\smallsetminus C$. If for some $m$, with $n_{0}+1\leq m\leq n_{0}+\lfloor\frac{t}{2}\rfloor$, $|X_{m}\cap C|\geq2$, then there exists $m_{0}$, with $n_{0}+1\leq m_{0}\leq n_{0}+\lfloor\frac{t}{2}\rfloor$ such that $X_{m_{0}}$ is disjoint from $C$.
Consequently, $X_{m_{0}}\sqcup\{x^{m_{0}+i}_{q_{m_{0}+i}}\in X_{m_{0}+i}\smallsetminus C: 1\leq i\leq k-|X_{m_{0}}|\}$ is disjoint from $C$. So for each $m$, with $n_{0}+1\leq m\leq n_{0}+\lfloor\frac{t}{2}\rfloor$,  $|X_{m}\cap C|=1$. Therefore for such case $C$ is a block of $\mathbb{G}(k,t)$ containing $X_{n_{0}}$. This implies that, for an arbitrary $k-$set $C$ which is not a block of $\mathbb{G}(k,t)$ then there exists a transversal $T(C)$ of $\mathbb{G}(k,t)$ which is disjoint from $C$.

Let $T$ be a transversal of $\mathbb{G}(k,t)$. Since $k\geq t+1$ so for each $n$, with $0\leq n\leq t-1$,  $|X_{n}|=k$ and $|T|=t$; then $X_{n}\smallsetminus T$ is non empty and choose $x^{n}_{q_{n}}\in X_{n}\smallsetminus T$. As we argued in the previous para that we have for each $n$, with $0\leq n\leq t-1$, $|X_{n}\cap T|\leq1$. If for some $m$, with $0\leq m\leq t-1$, $|X_{n}\cap T|<1$ i.e. $X_{m}$ is disjoint from $T$, then $X_{m}\sqcup\{x^{m+i}_{q_{m+i}}\in X_{m+i}\smallsetminus T: 1\leq i\leq k-|X_{m}|\}$ is disjoint from $T$, a contradiction. Hence the second part of the result follows.
\end{proof}

\begin{proposition}\label{F_closure}
Suppose $\tr{\mathbb{F}(k,t)}=t$ for $k\geq t$. Then for $k\geq t+1$, $\mathbb{F}(k,t)$ is a $\CIF{k,t}$.
\end{proposition}

\begin{proof}
Let $C$ be a blocking $k-$set of $\mathbb{F}(k,t)$ but $C\notin\mathbb{F}(k,t)$. It is enough to show that there exists at least one $T\in\mathbb{F}^{\top}(k,t)$ disjoint from $C$. If for each integer $n$, with $0\leq n\leq t-1$, there exists at least one $x_{n}\in X_{n}$ such that $x_{n}\notin C$, then $\{x_{n}:0\leq n\leq t-1\}$ is the required $T$ and we are done for this case. Suppose there exists at least one integer $n$, with $0\leq n\leq t-1$, such that $X_{n}\subsetneqq C$. Notice that for each $m$ with $m\neq n$ and $0\leq m\leq t-1$, there exists at least one $x_{m}\in X_{m}$ such that $x_{m}\notin C$. (Suppose it is false, then there exists at least one such integer $m$ with $X_{m}\subsetneqq C$. This implies that $X_{n}\sqcup X_{m}\subset C$, a contradiction arises since $k\geq t+1$.) When $t=2r-1$, then without loss of generality we can assume $X_{0}\subset C$. When $t=2r$, then without loss of generality we can assume either $X_{0}\subset C$ or $X_{\lfloor\frac{t-1}{2}\rfloor+1}=X_{r}\subset 
C$.

\noindent{\tt {Case~A :}} Let $X_{0}\subsetneqq C$. 

Here $C=X_{0}\sqcup Y$. We observe that if $Y$ is disjoint from $T_{n}:=\{x^{n}_{i}:0\leq i\leq n\}$, for some $n$ with $1\leq n\leq \lfloor\frac{t}{2}\rfloor$, then $T_{n}\sqcup\{x_{i}: x_{i}\in X_{i}\smallsetminus C,n+1\leq i\leq t-1\}$ is the required transversal disjoint from $C$ and we are done. So we assume that $Y\cap T_{n}\neq\emptyset$ for each $n$ with $1\leq n\leq \lfloor\frac{t}{2}\rfloor$. Since $|Y|=\lfloor\frac{t}{2}\rfloor$ and $T_{n}$, $1\leq n\leq \lfloor\frac{t}{2}\rfloor$, is $\lfloor\frac{t}{2}\rfloor$ pairwise disjoint sets so $Y$ intersects $T_{n}$ in exactly one point. Since $C\notin\mathbb{F}(k,t)$ so $Y$ is not of the form $\{x^{i}_{p_{i}}:1\leq i\leq \lfloor\frac{t}{2}\rfloor\}$. In the next para, under these assumptions on $Y$, we produce a transversal $T\in\mathbb{F}^{\top}(k,t)$ which is disjoint from both $Y$ and $X_{0}$. (Consequently, such a $T$ is disjoint from both $C$ and $X_{0}$.)

We have $|Y\cap\{x^{1}_{0}, x^{1}_{1}\}|=1$ suppose $x^{1}_{\epsilon_{1}}\in Y$ and $x^{1}_{1-\epsilon_{1}}\notin Y$, where $\epsilon_{1}\in\{0,1\}$. Construct $c_{1}=\epsilon_{1}$. If  $Y$ is disjoint from $\{x^{2}_{c_{1}}, x^{2}_{1+c_{1}}\}$, then
\begin{equation*}
\{x^{1}_{1-\epsilon_{1}},x^{2}_{c_{1}},x^{2}_{1+c_{1}}\}\sqcup\{x_{i}:x_{i}\in X_{i}\smallsetminus C,3\leq i\leq t-1\}
\end{equation*}
is the required transversal and we are done. So let 
$|Y\cap\{x^{2}_{c_{1}}, x^{2}_{1+c_{1}}\}|=1$ suppose $x^{2}_{c_{1}+\epsilon_{2}}\in Y$ and $x^{2}_{c_{1}+1-\epsilon_{2}}\notin Y$, where $\epsilon_{2}\in\{0,1\}$. Construct $c_{2}=c_{1}+\epsilon_{2}$. In general our construction procedure is as follows: suppose we already constructed a sequence $c_{1},c_{2},\ldots,c_{m}$ with the following properties.
\begin{enumerate}[\normalfont(a)]
\item For each $n$ with $1\leq n\leq m$, $c_{n}=c_{n-1}+\epsilon_{n}$ and $\epsilon_{n}\in\{0,1\}$.
\item $\{x^{n}_{c_{n}}:1\leq n\leq m\}\subset Y$.
\item $S_{m}:=\{x^{1}_{1-\epsilon_{1}}\}\sqcup\{x^{n}_{c_{n-1}+1-\epsilon_{n}}:2\leq n\leq m\}$ is disjoint from $Y$.
\end{enumerate} 
Now we construct $c_{m+1}$ if necessary. If  $Y$ is disjoint from $\{x^{m+1}_{c_{m}}, x^{m+1}_{1+c_{m}}\}$, then  
\begin{equation*}
S_{m}\sqcup\{x^{m+1}_{c_{m}}, x^{m+1}_{1+c_{m}}\}\sqcup\{x_{i}:x_{i}\in X_{i}\smallsetminus C, m+2\leq i\leq t-1 \}
\end{equation*}
is the required transversal and we are done. Now let $|Y\cap\{x^{m+1}_{c_{m}}, x^{m+1}_{1+c_{m}}\}|=1$, suppose $x^{m+1}_{c_{m}+\epsilon_{m+1}}\in Y$ and $x^{m+1}_{c_{m}+1-\epsilon_{m+1}}\notin Y$, where $\epsilon_{m+1}\in\{0,1\}$. Construct $c_{m+1}=c_{m}+\epsilon_{m+1}$. This yields 
$\{x^{n}_{c_{n}}:1\leq n\leq m+1\}\subset Y$ and $S_{m+1}$ is disjoint from $Y$. Since $Y$ is not of the form 
$\{x^{i}_{p_{i}}:1\leq i\leq \lfloor\frac{t}{2}\rfloor\}$ so this sequence contains at most 
$\lfloor\frac{t}{2}\rfloor-1$ terms. If this sequence contains exactly $M$ terms, then $Y$ is disjoint from $\{x^{M+1}_{c_{M}}, x^{M+1}_{1+c_{M}}\}$. Consequently, 
\begin{equation*}
S_{M}\sqcup\{x^{M+1}_{c_{M}}, x^{M+1}_{1+c_{M}}\}\sqcup\{x_{i}:x_{i}\in X_{i}\smallsetminus C, M+2\leq i\leq t-1 \}
\end{equation*}
is the required transversal.

\noindent{\tt {Case~B :}} Let $X_{\lfloor\frac{t-1}{2}\rfloor+1}\subsetneqq C$. 

Here $C=X_{\lfloor\frac{t-1}{2}\rfloor+1}\sqcup Y$. This case is similar to the above case. (Precisely, we need to replace $\lfloor\frac{t}{2}\rfloor$ by $\lfloor\frac{t-1}{2}\rfloor$, $x^{(\bullet)}_{p}$ by $x^{\lfloor\frac{t-1}{2}\rfloor+1+(\bullet)}_{p}$ and $x_{(\bullet)}$ by $x_{\lfloor\frac{t-1}{2}\rfloor+1+(\bullet)}$. ) 
\end{proof}

\section{Some applications}\label{beating}

In this section, it is shown that Example~\ref{FOT(4,2)} and Example~\ref{FOT(4&5,3)} are counter examples to \cite[\textsection~3, Conjecture~4]{MR1383503} in special cases. In the following examples we continue with the notation of Construction~\ref{circular_construction}.

\begin{example}\label{FOT(4,2)}
It is easy to check that for $k\geq2$, $\tr{\mathbb{F}(k,2)}=2$. So by Proposition~\ref{F_closure} we have, for $k\geq3$, $\mathbb{F}(k,2)$ is a $\CIF{k,2}$. We observe that for $0\leq p\leq k-2$ and $0\leq q\leq k-1$, the transversals of $\mathbb{F}(k,2)$ are  $\{x^{0}_{p},x^{1}_{q}\}$; $\{x^{1}_{0},x^{1}_{1}\}$. Hence there are $k^2-k+1$ number of transversals and $2$ blocks in $\mathbb{F}(k,2)$. So if we plug in $k=4$ we have a $\CIF{4,2}$ with $2$ blocks and $13$ transversals. Let $\mathcal{A}$ be the unique $\MIF{2}$ isomorphic to $\{\{a,b\},\{b,c\},\{a,c\}\}$ and $P_{\mathcal{A}}\cap P_{\mathbb{F}(4,2)}=\emptyset$. Therefore by Theorem~\ref{embedding_theorem},  
$\mathbb{F}(4,2)\sqcup\mathcal{A}\circledast\mathbb{F}^{\top}(4,2)$ is a $\MIF{4}$ with $42$ blocks and $10$ points. In this $\MIF{4}$ there are $3$ points in $26$ blocks, $5$ points in $14$ blocks and $2$ points in $10$ blocks. 
\end{example}

\begin{example}\label{FOT(4&5,3)}
It is easy to check that for $k\geq3$, $\tr{\mathbb{F}(k,3)}=3$. So by Proposition~\ref{F_closure} we have, for $k\geq4$, $\mathbb{F}(k,3)$ is a $\CIF{k,3}$. We observe that for $0\leq p,q,r\leq k-2$, the transversals of $\mathbb{F}(k,3)$ are  $\{x^{0}_{p},x^{1}_{q},x^{2}_{r}\}$; $\{x^{0}_{0},x^{0}_{1},x^{1}_{p}\}$; $\{x^{1}_{0},x^{1}_{1},x^{2}_{p}\}$ and $\{x^{2}_{0},x^{2}_{1},x^{0}_{p}\}$. Hence there are $(k-1)^{3}+3(k-1)$ number of transversals and $6$ blocks in $\mathbb{F}(k,3)$. So if we plug in $k=4$ and $k=5$ respectively, we have a $\CIF{4,3}$ and $\CIF{5,3}$ with $6$ blocks and $36$ \& $76$ transversals respectively. Let $\mathcal{A}$ be the unique $\MIF{1}$ (respectively, unique $\MIF{2}$ isomorphic to $\{\{a,b\},\{b,c\},\{a,c\}\}$) and $P_{\mathcal{A}}\cap P_{\mathbb{F}(4,3)}=\emptyset$(respectively, $P_{\mathcal{A}}\cap P_{\mathbb{F}(5,3)}=\emptyset$). By Theorem~\ref{embedding_theorem}, 
$\mathbb{F}(4,3)\sqcup\mathcal{A}\circledast\mathbb{F}^{\top}(4,3)$  is a $\MIF{4}$ with $42$ blocks
(respectively, $\mathbb{F}(5,3)\sqcup\mathcal{A}\circledast\mathbb{F}^{\top}(5,3)$ is a $\MIF{5}$ with $234$ blocks). In this $\MIF{4}$ there are $1$ point in $36$ blocks, $6$ points in $16$ blocks and $3$ points in $12$ blocks.
\end{example}

\begin{remark}
Example~\ref{FOT(4,2)} and Example~\ref{FOT(4&5,3)} proves existence of two non isomorphic $\MIF{4}$ with $42$ blocks. It disproves a special case (case $k=4$) of Conjecture~4 in \cite{MR1383503}, which claims such $\MIF{4}$ is unique up to isomorphism.
\end{remark}

\begin{remark}
Example~\ref{FOT(4&5,3)} proves existence of a $\CIF{5,3}$ namely $\mathbb{F}(5,3)$ with $6$ blocks and $76$ transversals. By Corollary~\ref{c_embedding_theorem} we have $\M{5}\geq 234$. It disproves a special case (case $k=5$) of Conjecture~4 in \cite{MR1383503}, which claims $\M{5}=228$.
\end{remark}

For large positive integer $k$, any $\MIF{k}$ with $\M{k}$ blocks contains at least (approximately) $(e-1)k!$ blocks. Disproving this was the prime object of article \cite{MR1383503}. Here we present an alternative simpler construction to prove $\M{k}$ is at least $(\frac{k}{2})^{k-1}$.

\begin{theorem}\label{LovaszConjecture}
Let $k\geq t+1$. Then
\begin{align}\label{M(k)}
\M{k}\geq\left\{\begin{array}{lcl}
                (2r-1)(k-r+1)^{r-1}+(k-r+1)^{2r-1}\M{k-2r+1} & \textnormal{if} & t=2r-1\\
                2r(k-r)^{r-1}+(k-r)^{r}(k-r+1)^{r}\M{k-2r} & \textnormal{if} & t=2r.
            \end{array}\right.
\end{align}
\end{theorem}
\begin{proof}
Let $\mathcal{A}$ be a $\MIF{k-t}$ with $\M{k-t}$ blocks. By Theorem~\ref{G_closure} and Theorem~\ref{embedding_theorem}, it follows that $\mathbb{G}(k,t)\sqcup\mathcal{A}\circledast\mathbb{G}^{\top}(k.t)$ is a $\MIF{k}$. Here we observe that any block of $\mathbb{G}(k,t)$ is of the form 
\begin{equation*}
X_{n}\sqcup\{x^{n+i}_{p}\in X_{n+i}:1\leq i\leq k-|X_{n}|\},
\end{equation*}
where $0\leq n\leq t-1$. It means that for each $X_{n}$, with $0\leq n\leq t-1$, and for each $i$, with $1\leq i\leq k-|X_{n}|$, there are $|X_{n+i}|$ number of choices for $x^{n+i}_{p}$. Therefore  there are $\overset{k-|X_{n}|}{\underset{i=1}\prod}|X_{n+i}|$ choices for such blocks. Hence,
\begin{align*}
|\mathbb{G}(k,t)|\geq\left\{\begin{array}{lcl}
                (2r-1)(k-r+1)^{r-1} & \textnormal{if} & t=2r-1\\
                2r(k-r)^{r-1} & \textnormal{if} & t=2r.
            \end{array}\right.
\end{align*}
Also by using Theorem~\ref{G_closure}, we have
\begin{align*}
|\mathbb{G}^{\top}(k,t)|=\left\{\begin{array}{lcl}
                (k-r+1)^{2r-1} & \textnormal{if} & t=2r-1\\
                (k-r)^{r}(k-r+1)^{r} & \textnormal{if} & t=2r.
            \end{array}\right.
\end{align*}
Therefore the results follows from Corollary~\ref{c_embedding_theorem}. 
\end{proof}

If we plug in $t=k-1$ in \eqref{M(k)}, we have the following corollary (Theorem~1 of \cite[\textsection~2]{MR1383503}), which shows that $\M{k}$ grows like at least $(\frac{k}{2})^{k-1}$ and it disproves Lov\'{a}sz Conjecture.

\begin{corollary}[{Frankl-Ota-Tokushige}]\label{beating_LovaszConjecture}
 \begin{align*}
\M{k}\geq\left\{\begin{array}{lcl}
               (\frac{k}{2}+1)^{k-1} & \textnormal{if} & k \textup{ is even }\\
                (\frac{k+1}{2})^{\frac{k-1}{2}}(\frac{k+3}{2})^{\frac{k-1}{2}} & \textnormal{if} & k \textup{ is odd }.
            \end{array}\right.
\end{align*}
\end{corollary}

The main problem of interest is to find a $\MIF{k}$ with $\M{k}$ blocks. Using Theorem~\ref{embedding_theorem}, we observe that this problem actually boils down to find a $\CIF{k,t}$ $\mathcal{F}$ and a $\MIF{k-t}$ $\mathcal{A}$ so that $|\mathcal{F}|+|\mathcal{A}||\mathcal{F}^{\top}|$ is maximum for some suitable choice of $t\leq k-1$. So we formulate the following conjecture and close this article.

\begin{conjecture}
For large positive integer $k$, any $\MIF{k}$ with $\M{k}$ blocks contains a $\CIF{k,t}$, for some $t\leq k-1$, which is isomorphic to a subfamily of $\mathbb{G}(k,t)$.
\end{conjecture}

\begin{acknowledgement}
It is author's radiant sentiment to place on record his sincere gratitude to Professor Bhaskar Bagchi, for the guidance which were extremely valuable to study this subject area theoretically.
\end{acknowledgement}

\end{document}